\documentclass[graybox]{svmult}

\usepackage{mathptmx}       
\usepackage{helvet}         
\usepackage{courier}        
\usepackage{type1cm}        
\usepackage{makeidx}         
\usepackage{graphicx}        
\usepackage{multicol}        
\usepackage[bottom]{footmisc}

\usepackage{graphicx}
\usepackage{amscd}
\usepackage{amsmath}
\usepackage{amsfonts}
\usepackage{psfrag}
\usepackage{amssymb}
\usepackage{verbatim}
\usepackage{comment}
\usepackage{mathrsfs}

\newcommand{\Lb}{\mbox {\boldmath ${\Lambda}$}}  
 \newcommand{\Lbs}{\mbox{\scriptsize\boldmath ${\Lambda}$}}
 
  \newcommand{\Pb}{\mbox {\bf P}} 
   \newcommand{\Pbs}{\mbox{\scriptsize\bf P}}

\def\FrP{{\mathfrak P}}   
\def\Vol{{\rm Vol}}   

\def\e{\epsilon}

\def\freq{{\rm freq}}

\def\dist{{\rm dist}}
\def\supp{{\rm supp}}

\def\Xk{\Xi}
\def\Ok{{\mathcal O}}

\newcommand{\lam}{\lambda}
\def\Lam{\Lambda}
\newcommand{\gam}{\gamma}
\newcommand{\om}{\omega}

\def\Om{\Omega}

\newcommand{\by}{{\bf y}}

\def\bbe{{\bf e}}
\newcommand{\R}{{\mathbb R}}

\newcommand{\Z}{{\mathbb Z}}
\newcommand{\C}{{\mathbb C}}

\def\N{{\mathbb N}}
\def\Compl{{\mathbb C}}

\def\bo{{\bf 0}}
\def\bu{{\bf u}}
\def\bw{{\bf w}}
\def\bz{{\bf z}}

\def\Dk{{\mathcal D}}

\def\Sk{{\mathcal S}}

\def\Pk{{\mathcal P}}

\def\Sf{{\sf S}}

\def\bx{{\mathbf x}}
\def\bv{{\mathbf v}}
\def\be{\begin{equation}}
\def\ee{\end{equation}}

\newcommand{\es}{\emptyset}

\def\ov{\overline}

\def\ve1{\vec{1}}

\def\Tk{{\mathcal T}}
\def\Ak{{\mathcal A}}

\def\wt{\widetilde}

\makeindex             

\begin{document}

\title*{Delone sets and dynamical systems}
\author{Boris Solomyak}
\institute{Boris Solomyak \at Bar-Ilan University, Ramat-Gan, Israel; \email{bsolom3@gmail.com}}
%
%
\maketitle

\abstract*{Each chapter should be preceded by an abstract (10--15 lines long) that summarizes the content. The abstract will appear \textit{online} at \url{www.SpringerLink.com} and be available with unrestricted access. This allows unregistered users to read the abstract as a teaser for the complete chapter. As a general rule the abstracts will not appear in the printed version of your book unless it is the style of your particular book or that of the series to which your book belongs.
Please use the 'starred' version of the new Springer \texttt{abstract} command for typesetting the text of the online abstracts (cf. source file of this chapter template \texttt{abstract}) and include them with the source files of your manuscript. Use the plain \texttt{abstract} command if the abstract is also to appear in the printed version of the book.}

\abstract{In these notes we focus on selected topics around the themes: Delone sets as models for quasicrystals, inflation symmetries and expansion constants,
substitution Delone sets and tilings, and associated dynamical systems.
}

\section{Introduction}
This is an expository article, based on a 1.5-hour talk delivered at the School on Tilings and Dynamics, held in CIRM, Luminy (November 2017). There are no new results here; most of the content is at least 10 years old. The goal of the talk was to supplement the lectures of N. P. Frank \cite{FrankLN}, and the notes are written in the same spirit. There is, of course, too much material here for a single lecture, but we tried to present at least a somewhat complete picture of the subject, from a personal viewpoint. Two themes that we tried to emphasize are (a) the characterization of expansion symmetries, and (b) the duality of substitution Delone sets and substitution tilings.

In Sections 2-4 we develop the notions of Delone sets and their classification, as well as characterization of (scalar) inflation symmetries. We largely follow the paper by J. C. Lagarias \cite{Lag99},
although many original ideas are due to Y. Meyer \cite{Meyer72,Meyer95}. Some relatively easy statements are  proved completely;  other proofs are sketched, and still others are referred to the original papers. Several examples are worked out in detail to illustrate the results.

Section 5 is central for us, although it contains almost no proofs. Here we develop the notions of substitution Delone sets and substitution tilings and discuss the relation between them. In particular, {\em representable} substitution Delone sets (or rather, ``multi-color'' Delone sets, or $m$-sets) are introduced. Much of the section is based on the paper by J. C. Lagarias and Y. Wang \cite{LaWa}, with a few important additions from \cite{LMS1}. We include a discussion of a theorem about pseudo-self-affine tilings being mutually locally derivable with self-affine tilings, following \cite{SolPseudo}, since this is a nice application of the duality between substitution Delone sets and tilings.

The last Section 6 is devoted to dynamical systems arising from tilings and Delone sets. Since this is one of the main themes of \cite{FrankLN}, we have limited ourselves to a very brief exposition, highlighting the aspects related to the topics developed in earlier sections, specifically, the number-theoretic issues arising from the description of eigenvalues and the Meyer property.


\section{Delone sets of finite type}

Here we borrow much from the paper \cite{Lag99}.

\begin{definition}{
A discrete set $X$ in $\R^d$ is a {\em Delone set} if it is

(a) {\em Uniformly discrete:} there exists $r>0$ such that every open ball of radius $r$ contains at most one point of $X$; equivalently, the distance between distinct points in $X$ is at least $2r$.

(b) {\em Relatively dense:} there exists $R>0$ such that every closed ball of radius $R$ contains at least one point of $X$.
}
\end{definition}

Delone sets are also sometimes called {\em separated nets}. The notion of Delone sets as fundamental objects of study in crystallography was introduced by the Russian school in the 1930's; in particular, by Boris Delone (or Delaunay). One can think about a Delone set as an idealized model of an atomic structure of a material without ``holes''. This is obviously too general to be considered an ``ordered'' structure, so we impose some conditions on the set.

For a set $X\subset \R^d$ denote by $[X]$ the abelian group (equivalently, $\Z$-module) generated by $X$. Recall  that $[X] := \bigl\{\sum_{i=1}^k n_i x_i:\, k\in \N,\ n_i\in \Z\bigr\}$.
The following  classes of Delone sets have been studied:

\begin{definition} {
Let $X$ be a Delone set in $\R^d$.

(a) $X$ is {\em finitely generated} if $[X-X]$ (equivalently, $[X]$) is a finitely generated abelian group.

(b) $X$ is of {\em finite type} if $X-X$ is a discrete closed subset of $\R^d$; i.e., the intersection of $X-X$ with any ball is a finite set.

}
\end{definition}

In the next section we will add to these the notion of {\em Meyer set}.

\begin{exercise}[easy]
Give an example of a finitely generated Delone set which is not of finite type.
\end{exercise}

We will see many such examples in Section 3.

\begin{exercise}[easy]
Show that a Delone set $X$ is of finite type if and only if  it has {\em finite local complexity} (FLC), that is, for any $\rho>0$ there are finitely many ``local clusters'' $X\cap B(\by,\rho)$, up to translation.
\end{exercise}

We will prove that every Delone set of finite type is finitely generated, but first we 
introduce the important notion of {\em address map}. Recall that a finitely generated subgroup of $\R^d$ is free. We can choose a free basis, its cardinality is independent of the choice of basis; it is called the {\em rank} of the free group.

\begin{definition}{ 
Let $X$ be a finitely generated Delone set. Choose a basis of $[X]$, say,
$$
[X] = [\bv_1,\bv_2,\ldots,\bv_s].
$$
The address map $\phi:[X]\to \Z^s$ associated to this basis is
$$
\phi\Bigl( \sum_{i=1}^s n_i \bv_i \Bigr) = (n_1,\ldots,n_s).
$$
}
\end{definition}

The address map for the Delone set of control points of a self-similar tiling has been used by Thurston \cite{Thur} (we describe it in Section 4).
The address map is not unique, it depends on the choice of a basis, so it is determined up to left-multiplication by an element of of $GL(s,\Z)$. 
Observe that $s\ge d$ since the linear span of a Delone set is obviously the entire $\R^d$. If $s=d$, then $X$ is a subset of a {\em lattice};  this is a special case --- usually, $s>d$. 
 Quoting J. Lagarias \cite{Lag99}:
\begin{quote}The structure of a finitely generated Delone set is to some extent analyzable by studying its image in $\R^s$ under the address map. The address map describes $X$ using $s$ ``internal dimensions.''
\end{quote}

\begin{theorem}[see \cite{Lag99}] \label{Del-Lip}
For a Delone set $X$ in $\R^d$, the following properties are equivalent:

{\rm (i)} $X$ is a Delone set of finite type.

{\rm (ii)} $X$ is finitely generated and any address map $\phi:\,[X]\to \Z^s$ is {\bf globally Lipshitz on $X$}:
$$
\|\phi(\bx) - \phi(\bx') \|\le C_0 \|\bx-\bx'\|,\ \ \mbox{for all}\ \bx,\bx'\in X,
$$
for some constant $C_0$ depending on $\phi$.
\end{theorem}

\begin{proof}[sketch] (ii) $\Rightarrow$ (i). If $\bx,\bx'\in X$ are such that $\|\bx-\bx'\|\le T$, then by hypothesis,
$$
\|\phi(\bx-\bx')\| = \|\phi(\bx) - \phi(\bx')\| \le C_0 T.
$$
Addresses all lie in $\Z^s$, hence there are only finitely many choices for $\phi(\bx-\bx')$, hence for $\bx-\bx'$, since $\phi:X\to \Z^s$ is one-to-one. It follows that $(X-X)\cap B(\bo,T)$ is finite, so $X$ is of finite type. \qed

\begin{lemma}[see \cite{Lag99}] \label{lem-lag1}
Let $X$ be a Delone set in $\R^d$ with parameters $(r,R)$. Then there exist constants $k>1$ and $C = C(R,r)$, such that given any two points $\bx,\bx'\in X$ there is a chain of points
$$
\bx = \bx_0, \bx_1, \bx_2,\ldots,  \bx_m = \bx'
$$
in $X$ such that 

{\rm (a)} $\|\bx_i -\bx_{i-1}\|\le kR$ for all $i$;

{\rm (b)} $m \le C\|\bx - \bx'\|$. 
\end{lemma}

 Lagarias \cite{Lag99} proved that one can take $k=2$ and $C(R,r) = 4R/r^2$. We  give a simple proof for $k=4$ and $C(R,r) = (2R)^{-1} + r^{-1}$.
 
 \begin{sloppypar}
 \begin{proof}[Lemma \!\!\!\ref{lem-lag1}]
 Let $m= \lfloor \frac{\|\bx-\bx'\|}{2R}\rfloor + 1$ and consider the points $\bx=\bx_0', \bx_1',\ldots,\bx_{m-1}', \bx=
 \bx_m'$ at equal distance $\frac{\|\bx-\bx'\|}{m} \le 2R$ from each other. By the $(r,R)$-property, we can choose $\bx_i\in X$, for $i=1,\ldots,m-1$, such that $\|\bx_i - \bx_i'\| \le R$. Then $\|\bx_i - \bx_{i+1}\|\le R + 2R + R = 4R$ for all $i=0,\ldots,m-1$. It remains to note that
 $$
 m \le \frac{\|\bx-\bx'\|}{2R} + 1 \le \frac{\|\bx-\bx'\|}{2R} + \frac{\|\bx-\bx'\|}{r},
 $$
 and the proof is complete. \qed
 \end{proof}
 \end{sloppypar}

Now we can deduce the implication (i) $\Rightarrow$ (ii). By Lemma~\ref{lem-lag1}, $[X-X]$ is generated by $\{\bx-\bx':\ \|\bx-\bx'\|\le kR\}$, which is finite; thus $X$ 
 is finitely generated. Let $\phi:[X]\to \Z^s$ be an address map. Let
$$
C_1 := \max\{\|\phi(y)\|:\ y\in (X-X)\cap B(\bo, kR)\}.
$$
Given $\bx,\bx'$, by Lemma~\ref{lem-lag1} we have a chain in $X$ connecting $\bx$ to $\bx'$ which  satisfies $\|\bx_i-\bx_{i+1}\|\le kR$, and the number of points is at most $C\|\bx-\bx'\|$.
Using $\Z$-linearity of $\phi$ on $[X]$, we can write
\begin{eqnarray*}
\|\phi(\bx)-\phi(\bx')\| & \le & \sum_{i=1}^m \|\phi(\bx_i)- \phi(\bx_{i-1}\| \\
& = & \sum_{i=1}^m \|\phi(\bx_i-\bx_{i-1}\| \\
& \le & C_1 m \le C_1C\|\bx-\bx'\|,
\end{eqnarray*}
which proves (ii). \qed
\end{proof}

\section{Meyer sets}

Meyer sets were introduced by Yves Meyer \cite{Meyer95}  under the name ``quasicrystal''; now they are generally called Meyer sets, see \cite{Lag96,Moody97}. In fact, the concept was developed by Meyer much earlier, in the late 1960's, see \cite{Meyer72}. 

\begin{definition}{
Let $X$ be a Delone set in $\R^d$. It is called a {\em Meyer set} if the self-difference set $X-X$ is uniformly discrete, equivalently, a Delone set.
}
\end{definition}

Actually, this was not the original definition. An interesting feature of Meyer sets is that they can be characterized in seemingly very different terms: using discrete geometry, almost linear mappings, and cut-and-project sets. 

\begin{definition} {
Let $\R^n = E^d \oplus H$, where $E^d\approx \R^d$ (the ``physical space'') and $H\approx \R^m$ (the ``internal space) are  linear subspaces of $\R^n$, not necessarily orthogonal. Let $\pi$ and $\pi_{\rm int}$ be the projection onto
 $E^d$ parallel to $H$, and the projection to $H$ parallel to $E^d$, respectively. A {\em window}  $\Om$ is a bounded open subset of $H$.
Let $\Lam$ be a full rank lattice in $\R^n$
(that is, $\Lam$ is a discrete subgroup of rank $n$; equivalently, it is a subgroup which forms a Delone set). The {\em cut-and-project set} $X(\Lam,\Om)$ associated to the data $(\Lam,\Om)$ is
$$
X(\Lam,\Om) = \pi \bigl(\{\bw\in \Lam:\ \pi_{\rm int} (\bw) \in \Om\}\bigr).
$$
A cut-and-project set is called {\em nondegenerate} if $\pi$ is one-to-one on $\Lam$. It is {\em irreducible} if $\pi_{\rm int}(\Lam)$ is dense in $H$.
Cut-and-project sets (sometimes with different requirements for the window) are also called {\em model sets}.
}
\end{definition}

\begin{theorem}[Y. Meyer, J. Lagarias, see \cite{Lag99}] \label{th-meyer}
For a Delone set $X$ in $\R^d$, the following properties are equivalent:

{\rm (i)} $X$ is a Meyer set, that is, $X-X$ is a Delone set.

{\rm (ii)} there is a finite set $F$ such that $X-X \subseteq X+F$.

{\rm (iii)} $X$ is a finitely generated Delone set and the address map $\phi:[X]\to \Z^s$ is an {\em almost linear mapping},  i.e., there is a linear map $L:\R^d\to \R^s$ and $C>0$ such that
$$
\|\phi(\bx) - L\bx\|\le C\ \ \mbox{for all}\ \bx\in X.
$$

{\rm (iv)} $X$ is a subset of a non-degenerate cut-and-project set.
\end{theorem}

The theorem is mostly due to Meyer, except that the implication (i) $\Rightarrow$ (ii) was proved by Lagarias \cite{Lag96}. 
R. V. Moody \cite{Moody97} did much to popularize the concept; his article \cite{Moody97} contains other interesting characterizations of Meyer sets, e.g., in terms of Harmonic Analysis and a certain kind of ``duals".

\begin{proof}[partial sketch] 
 (ii) $\Rightarrow$ (iii) First note that $X$ is of finite type. Indeed, $(X-X)\cap B(\bo,N) \subset (X+F)\cap B(\bo,N)$ is finite for any $N>0$. Thus it is finitely generated, and we can consider the address map $\phi$. We construct $L:\,\R^d\to \R^s$ as an ``ideal address map,'' as follows. For each $\by\in \R^d$ define
$$
L(\by):= \lim_{k\to \infty} \frac{\phi(\bx_k)}{2^k},
$$
where $\bx_k\in X$ satisfies $\|\bx_k - 2^k \by\|\le R$. Using (ii) one proves that this limit exists and is unique (independent of the choice of $\bx_k$), and is linear, roughly along these lines:
By the definition of $\bx_k,\bx_{k+1}$ and the triangle inequality we have $\|2\bx_k - \bx_{k+1}\|\le 3R$. But $$2\bx_{k} - \bx_{k+1} = \bx_k - (\bx_{k+1}-\bx_k)  = \bx_k - (\bz + \bw),$$
for some $\bz\in X$ and $\bw\in F$ (the finite set from (ii)). Thus,
$$
\|\bx_k - \bz\| \le 3R + C_1,\ \ \mbox{where}\ C_1 = \max_{\bu \in F} \|\bu\|.
$$
Since $\phi$ is $\Z$-linear on $X$, we have
$$
2\phi(\bx_k) - \phi(\bx_{k+1}) = [\phi(\bx_k) - \phi(\bz)] - \phi(\bw),
$$
hence
\begin{eqnarray*}
\|2\phi(\bx_k) - \phi(\bx_{k+1})\|  & \le & \|\phi(\bx_k) - \phi(\bz)\| + \|\phi(\bw)\|\\[1.2ex]
& \le & C_0(3R+C_1) + C_2=:C'.
\end{eqnarray*}
where $C_2 = \max_{\bu\in F} \|\phi(\bu)\|$ and $C_2$ is from Theorem~\ref{Del-Lip}.
Therefore,
$\|\frac{\phi(\bx_k)}{2^k} - \frac{\phi(\bx_{k+1})}{2^{k+1}}\|\le C'/2^{k+1}$, and convergence follows. Now it is not difficult to check linearity of $L$. For $\bx=\by\in X$ we can take $\bx_0 = \bx$
and, summing the geometric series, obtain
$$
\|\phi(\bx) - L(\bx)\| \le C'.
$$
This shows almost linearity of the address map.

(iii) $\Rightarrow$ (i) It is clear that $X-X$ is relatively dense, so we only need to show that it is uniformly discrete. Equivalently, that there is a lower bound on the norm of 
$\bz\in (X-X)-(X-X)$, whenever $\bz\ne \bo$. Suppose that $\|\bz\|\le R$. By the hypothesis (iii), we have $\|\phi(\bz)-L\bz\| \le 4C$, since $\phi$ is $\Z$-linear on $[X]$ and $L$ is linear on $\R^d$. Therefore,
$$
\|\phi(\bz)\|\le 4C + \|L\|R.
$$
Since $\phi$ is one-to-one on $[X]$ and $\phi(\bz)\in \Z^s$, we see that there are only finitely many possibilities for $\bz$. The desired claim follows.

For the equivalence of (iv) and (i)-(iii), see \cite{Lag99} or \cite{Moody97}. The direction (iv) $\Rightarrow$ (i) is simple: one needs to verify that a relatively dense subset of a cut-and-project set is Delone, then note that for a lattice $\Lam$ and a window $\Om$ we have $X(\Lam,\Om)-X(\Lam,\Om) \subseteq X(\Lam,\Om-\Om)$. 

The direction (iii) $\Rightarrow$ (iv) proceeds by taking the physical space $E^d=L(\R^d) \subset \Z^s$ and a complementary subspace $H$ as the internal subspace. We then ``lift'' $X$ to $\Z^s$ via the address map (which is, of course, a lattice in $\R^s$), project it onto $H$ and note that the projection is bounded by the property (iii). We can then choose  any bounded open set containing the projection as a window. See \cite{Lag99} for details. \qed
\end{proof}

\section{Inflation symmetries}

This section is largely based on \cite{Lag99}.

\begin{definition}{
A Delone set $X$ in $\R^d$ has an inflation symmetry by the real number $\eta>1$ if $\eta X\subseteq X$.
}
\end{definition}

Recall that a (complex) number $\eta$ is an algebraic integer if $p(\eta)=0$ for some monic polynomial $p\in \Z[x]$, that is, $p(x) = x^s + \sum_{j=0}^{s-1} c_j x^j$, with $c_j\in \Z$.
The degree of $\eta$ is the minimal degree of $p(x) \in \Z[x]$ such that $p(\eta)=0$. The algebraic (or Galois) conjugates of $\eta$ are the other roots of the minimal polynomial for $\eta$. Several classes of algebraic integers appear:

\begin{definition}
{ Let $\eta$ be a real algebraic integer greater than one.

(a) $\eta$ is a {\em Pisot number} or {\em Pisot-Vijayaraghavan (PV)}-number if all algebraic conjugates satisfy $|\eta'|<1$.

(b) $\eta$ is a {\em Salem number} if for all conjugates $|\eta'|\le 1$ and at least one satisfies $|\eta'|=1$.

(c) $\eta$ is a {\em  Perron number} if for all conjugates $|\eta'|< \eta$.

(d) $\eta$ is a {\em Lind number} if for all conjugates $|\eta'|\le \eta$ and at least one satisfies $|\eta'| = \eta$.
}
\end{definition}

``Lind numbers'' were introduced by Lagarias \cite{Lag99}, but  this apparently  did not become standard terminology.

\begin{definition}
{
A Delone set $X$ in $\R^d$ is called {\em repetitive} if for any $T>0$ there exists $M_X(T)>0$ such that every ball of radius $M_X(T)$ contains a translated copy of every $X$-cluster of radius $T$.
}
\end{definition}

\begin{theorem}[J. Lagarias \cite{Lag99}, Y. Meyer  \cite{Meyer72}] \label{th-infl}
Let $X$ be a Delone set in $\R^d$ such that $\eta X\subseteq X$ for a real number $\eta>1$.

{\rm (i)} If $X$ is finitely generated, then $X$ is an algebraic integer.

{\rm (ii)} If $X$ is a Delone set of finite type, then $\eta$ is a Perron number or a Lind number.

{\rm (ii$'$)} If $X$ is {\em repetitive}  Delone set of finite type, then $\eta$ is a Perron number.

{\rm (iii)} If $X$ is a Meyer set, then $\eta$ is a Pisot number or a Salem number.

\end{theorem}

\begin{proof}[partial sketch] (i) We will  prove a more general statement:

\begin{lemma} \label{lem-algeb}
Let $X$ be a finitely generated Delone set in $\R^d$ such that $QX \subseteq X$ for some expanding linear map $Q$. Then all eigenvalues of $Q$ are algebraic integers.
\end{lemma}

\begin{proof}[of the lemma]
Since $X$ is finitely generated, we have $[X]=[\bv_1,\ldots,\bv_s]$ for some free generators and the address map $\phi: [X] \to \Z^s$. Since $Q X\subseteq X$, we also have $Q([X])\subseteq [X]$. It follows that $Q\bv_j$ is an integer linear combination of the vectors $\bv_j$ (recall that these are the free generators of $[X]$). Define the matrix $V = [\bv_1,\ldots,\bv_s]$ of size $d\times s$. We thus obtain an integer matrix $M$ of size $s\times s$ such that 
\be \label{imp}
Q V = VM.
\ee
It is clear that $\{\bv_j\}_{j\le s}$ spans $\R^d$, because $X$ does, hence ${\rm rank}(V) = d$. Let $\bbe$ be a left eigenvector of $Q$ corresponding to an eigenvalue $\lam$. Then $\lam\bbe V = \bbe Q V = \bbe VM$. Notice that $\bbe V$ is not zero, because the rows of $V$ are linearly independent. Thus
$\bbe V$ is an eigenvector for $M$ corresponding to $\lam$. But $M$ is an integer matrix, so all its eigenvalues are algebraic integers.\qed
\end{proof}

(ii)  Let $\gam$ be a conjugate of $\eta$. We continue the argument of the lemma. Since $\eta$ is an eigenvalue of the integer matrix $M$, so is $\gam$. Thus there is an eigenvector $\bbe_\gam\in \R^s$ corresponding to $\gam$. We want to prove that $|\gam|\le \eta$.  Let $\phi$ be the address map as in the proof of the lemma. Then $\{\phi(\bv_j)\}_{j\le s}$ is the canonical basis of $\R^s$, by definition. Since $\bv_j$'s are the generators of $[X]$, we must have that $\phi(X)$ spans $\R^s$. It follows that we can find $\bx_0, \bx \in X$ such that $\phi(\bx-\bx_0)=\phi(\bx)-\phi(\bx_0)$ has a non-zero coefficient corresponding to $\bbe_\gam$ in the canonical eigen(root)vector expansion corresponding to $M$. We have $\eta^n \bx\in X$ for all $n\in \N$, and 
\be \label{nunu}
\phi(\eta^n \bx) = M^n \phi(\bx),
\ee
by the definition of $M$. 
Now, by Theorem~\ref{Del-Lip}, we have, 
$$
\|\phi(\eta^n \bx)-\phi(\eta^n\bx_0)\| \le C_0\|\eta^n\bx-\eta^n\bx_0\| = C_0\eta^n \|\bx-\bx_0\|.
$$
On the other hand, by  (\ref{nunu}) and the choice of $\bx,\bx_0$, 
$$
\|\phi(\eta^n \bx)-\phi(\eta^n\bx_0)\| = \|M^n(\phi(\bx-\bx_0))\|\ge c|\gam|^n,
$$
for some $c>0$, and we can conclude that $|\gam| \le |\eta|$.\qed

(ii$'$) This can derived be similarly to  \cite[\S 10]{Thur}, see also \cite{Solnotes}; we omit the proof.

(iii) Consider the address map $\phi$ and the matrix $M$ as above. Recall that for a Meyer set the address map is almost linear, that is, there exists a linear map $L:\R^d\to \R^s$ such that 
$\|\phi(\bx) - L\bx\|\le C$ for all $\bx\in X$. Consider the range $H:= L(\R^d$), a $d$-dimensional subspace of $\R^s$. 

\medskip

 \noindent {\bf Claim.} {\em The subspace $H$ is invariant under $M$; in fact, $H$ is contained in the eigenspace of $M$ corresponding to $\eta$.}


\begin{proof}[of the claim]
Choose any unit vector $\bw\in H$, then $\bw=L\bz$ for some $\bz\in \R^d$. For any $k>1$ we can find $\bx\in X$ such that $\|k\bz-\bx\|\le R$.
Then $\|k\bw - L\bx\|\le \|L\|R$, hence 
\be \label{tr1}
\|k\bw - \phi(\bx)\|\le C+\|L\|R.
\ee
Since $\eta \bx\in X$, we also have $\|\phi(\eta \bx) - L(\eta \bx)\|\le C$, therefore, 
$$
\|\phi(\eta\bx) - \eta \phi(\bx)\|\le \|\phi(\eta\bx) - L(\eta\bx)\| + \eta\|\phi(\bx) - L\bx\|\le C(1+\eta).
$$
Recall from (\ref{nunu}) that $\phi(\eta \bx) = M\phi(\bx)$, thus $\|(M-\eta I)\phi(\bx)\|\le C(1+\eta)$. Combining this with (\ref{tr1}) yields
$$
\|(M-\eta I) k\bw\|\le C(1+\eta) + \|M-\eta I\|\cdot (C+\|L\|R)=: \wt{C}.
$$
Therefore, $\|(M-\eta I) \bw\|\le \wt{C}/k$ for all $k>1$, and the claim follows.
\end{proof}

Now we repeat part of the argument from (ii): given a conjugate $\gam$ of $\eta$ and the corresponding eigenvector $\bbe_\gam$ for $M$, choose $\bx$ with $\phi(\bx)$ having non-zero coefficient corresponding to  $\bbe_\gam$. By the Claim, $\bbe_\gam\not\in H$. As above, we have $\phi(\eta^n \bx) = M^n \phi(\bx)$, hence 
$$
c|\gam|^n \le \|M^n\phi(\bx)-\eta^nL\bx\| = \|\phi(\eta^n\bx) - L(\eta^n\bx)\|\le C,
$$
for some $c>0$, since $L\bx$ has a zero coefficient corresponding to  $\bbe_\gam$. 
It follows that $|\gam|\le 1$, as desired.\qed 
\end{proof}

\begin{example} \label{example} {
(i) Let $\eta>1$ be an irrational algebraic integer. One can  construct a finitely generated Delone set $X\subseteq \R$, with $\eta X\subseteq X$ as follows. We will have $X\subseteq \Z[\eta]$ (the ring generated by $\Z$ and $\eta$) and
$X=-X$. Start with $X\subset [0,\eta) = \{0,1\}$ and proceed by induction, adding points to $X$ from  $[\eta^k,\eta^{k+1})$, once we did this in $[\eta^{k-1},\eta^k)$. First make sure that
$X\cap [\eta^k,\eta^{k+1}) \supseteq \eta(X\cap [\eta^{k-1},\eta^k))$, and then add more points from $\Z[\eta]$, if necessary, to maintain the relative denseness, but also preserve uniform discreteness. The latter is easy, since $\Z[\eta]$ is dense in $\R$. Finally, observe that $\Z[\eta] = [1,\eta,\ldots,\eta^{s-1}]$ is finitely generated as a $\Z$-module, where $s$ is the degree of $\eta$.

\smallskip

(ii) Let $\eta>1$ be a Perron number. One can construct a Delone set $X\subseteq \R$ with inflation symmetry $\eta$. In fact, it will even be a {\em substitution Delone set}, discussed in the next section. It is obtained as a set of endpoints of a self-similar tiling of $\R$ corresponding to a primitive substitution. By a theorem of D. Lind \cite{Lind}, for any Perron  number $\eta$ there exists a primitive integer matrix $M$ with $\eta$ as a dominant eigenvalue. Moreover, by a minor modification of the construction, one can make sure that the entry $(1,1)$ of the matrix is positive and the first column sum is at least three, see \cite{Solnotes}. Then simply choose any substitution with substitution matrix $M$.

It is known that such an $X$ is Meyer if and only if $\eta$ is  a Pisot number.

\smallskip

(iii) {\bf $\beta$-integers.} 
Fix $\beta>1$, with $\beta\not\in\N$. Let $X_\beta = X_\beta^+\bigcup (-X_\beta^+)$, where
$$
X_\beta^+=\Bigl\{\sum_{j=0}^N a_j \beta^j,\ a_j \in \{0,1,\ldots,\lfloor \beta \rfloor\},\ \mbox{``greedy'' expansion} \Bigr\}.
$$
Then $X_\beta$ is relatively dense in $\R$ and $\beta X\subset X$.

\begin{itemize}
\item $X_\beta$ is Delone if and only if the orbit of 1 under $T_\beta(x) = \beta x$ (mod 1)  does not accumulate to 0.

\item Delone $X_\beta$ is finitely generated iff $\beta$ is an algebraic integer.

\item Delone $X_\beta$ is of finite type iff $\beta$ is a Parry $\beta$-number (see \cite{Parry}), i.e.,\ the orbit $\{T_\beta^n (1)\}_{n\ge 0}$ is finite.

\item If $\beta$ is Pisot (or Salem of degree four \cite{Boyd}), then $X_\beta$ is Meyer.
\end{itemize}

We will now explain the last claim, that Pisot $\beta$ implies the Meyer property, at the same time illustrating some of the concepts in the proofs of the theorem above.

Fix $\beta>1$ algebraic integer, such that $\{T_\beta^n 1\}_{n\ge 0}$ does not accumulate to zero. Then $X_\beta\subset \R$ is Delone and $[X_\beta] = \Z[\beta]$.
Free generators for $[X_\beta]$ can be chosen $v_j = \beta^{j-1}, \ j\le s$, where $s$ is the degree of $\beta$. Let $c_0 + c_1x + \cdots + c_{s-1}x^{s-1} + x^s$ be the minimal integer polynomial for $\beta$.


We have $Qx = \beta x$ on $\R$, and $QX_\beta \subset X_\beta$. 
Then $QV = VM$, where $V = [v_1,\ldots,v_s]$ (a $1\times s$ matrix), and
\begin{equation} \label{eq-matri}
M = \left(\begin{array}{ccccc} 0 & \ldots & \ldots & 0 & -c_0 \\
                                               1 & 0 & \ldots   &      0 & -c_1 \\
                                               0 & 1 & \ldots   &           0 & -c_2 \\
                                               \ldots & \ldots & \ldots  & \ldots & \ldots \\
                                               \ldots & \ldots &  \ldots & 0 & -c_{s-2} \\
                                               0 & 0 & \ldots & 1 & -c_{s-1} \end{array} \right)
\end{equation}
Let $\phi$ be the associated address map, $\phi: [X_\beta]=\Z[\beta]\to \R^s$. We have
$$
\phi(\beta^n) = M^n \phi(1) = M^n \left(\begin{array}{c} 1 \\ 0 \\ \vdots\\  0 \end{array} \right).
$$
Now suppose that $\beta$ is Pisot. Then we have
\begin{equation} \label{Pisot}
\phi(\beta^n) = \beta^n e_\beta + O(\varrho^n),
\end{equation}
where $e_\beta$ is the eigenvector of $M$ corresponding to $\beta$ and $\varrho \in (0,1)$ is the maximal absolute value of the Galois conjugates of $\beta$.
Define $L:\R\to \R^s$ by $L(x) = xe_\beta$, a linear map. 
We want to show that $\|\phi(x) - Lx\|\le C$ on $X_\beta$, whence $X_\beta$ is a Meyer set by Theorem~\ref{th-meyer}.
In view of (\ref{Pisot}), we have for $x=\sum_{j=0}^N a_j \beta^j\in X_\beta^+$:
\begin{eqnarray*}
\|\phi(x) - Lx\| & = & \left\|\phi\Bigl(\sum_{j=0}^N a_j \beta^j\Bigr) - L \Bigl(\sum_{j=0}^N a_j \beta^j\Bigr)\right\| \\
                      & = & O\Bigl( \max_j|a_j|\cdot \sum_{j=0}^N \varrho^j \Bigr) = O(1),
\end{eqnarray*}
as desired.

\smallskip

The same proof works, e.g., for the set of endpoints of a self-similar tiling on $\R$ with a Pisot inflation factor.

\medskip

(iv) {\bf Salem inflation factors.} For every Salem number $\beta$ there exists a Meyer set in $\R$ with inflation $\beta$; see \cite{Meyer72} for the original Meyer's construction. I am grateful to Shigeki Akiyama who showed me the following example, which is, apparently, ``folklore''.

Let $\beta$ be a Salem number of degree $s\ge 4$, and let $p(x) = c_0 + c_1 x + \cdots + c_{s-1} x^{s-1} + x^s$ be the minimal polynomial for $\beta$.
Let $\beta_2,\ldots,\beta_s$ be the Galois conjugates of $\beta$.
It is well-known that $p(x)$ is a reciprocal polynomial, i.e., $c_0=1$ and $c_{s-j} = c_j$ for $j=1,\ldots,s-1$, and the conjugates satisfy $|\beta_2| = \ldots = |\beta_{s-1}|=1$ and $|\beta_s|<1$, see \cite{Salem_book}.
Consider
$$
X_\beta:= \Bigl\{x = \sum_{n=0}^N a_n \beta^n:\ a_n\in \Z,\ \max_{2\le j \le s} \bigl|\sum_{n=0}^N a_n \beta_j^n \bigr| < 1\Bigr\}.
$$
It is immediate that $\beta X_\beta\subset X_\beta$. Moreover, $X_\beta$ is a non-degenerate model set, hence it is Meyer. 

In order to prove the last claim, we consider the matrix $M$ from (\ref{eq-matri}) above and identify $X_\beta$ with $X_\beta e_\beta$, where $e_\beta$ is the right eigenvector of $M$ corresponding to $\beta$. The subspace spanned by $e_\beta$ is our ``physical space''. The ``internal space'' $H$ is the linear span of the other eigenvectors. One can check that (with appropriate normalization) the coordinate $a_k$ of a vector $\by = (y_j)_1^s\in \R^s$ with respect to  the eigenvector
of $M$ corresponding to $\beta_k$ is given by $\sum_{j=0}^{s-1} y_j  \beta_k^j$. Thus, taking $\Z^s$ as a lattice and the window in $H$ given by the condition
$$
\|\by\|:= \max_{2\le k \le s} |a_k|<1\ \ \ \mbox{for}\ \ \ \by = \sum_{k=2}^s a_k e_{\beta_k},
$$
we get the desired representation as a cut-and-project set.
The details are left to the reader. (Instead of the unit ball in the $\ell^\infty$ norm as a window we can choose a different radius and a different norm to get other examples of Meyer sets with the same inflation symmetry.)
}
\end{example}


\section{Substitution Delone sets and substitution tilings}

This section is based on \cite{LaWa} and \cite{LMS03}; see also \cite{Bandt}.

\medskip

If we think about Delone sets as being models of atomic structures, it is natural to add the feature of ``color'' or ``type'' of a point/atom. Thus we are going to talk about ``$m$-sets''. (Sometimes, the term ``multiset'' is used, but it usually refers to a set with multiplicities, and we want to avoid this.)

\subsection{Substitution Delone $m$-sets}

\begin{definition}
An   $m$-multiset in $\R^d$ is a 
subset $\Lb = \Lam_1 \times \dots \times \Lam_m 
\subset \R^d \times \dots \times \R^d$ \; ($m$ copies)
where $\Lam_i \subset \R^d$. We also write 
$\Lb = (\Lam_1, \dots, \Lam_m) = (\Lam_i)_{i\le m}$.
We say that $\Lb=(\Lambda_i)_{i\le m}$ is a {\em Delone $m$-set} in $\R^d$ if
each $\Lambda_i$ is Delone and $\supp(\Lb):=\bigcup_{i=1}^m \Lambda_i 
\subset \R^d$ is Delone.
\end{definition}

Although $\Lb$ is a product of sets, it is convenient to think
of it as a set with types or colors, $i$ being the
color of points in $\Lambda_i$.  (However, we do not assume that $\Lam_i$ are pairwise disjoint!)
A {\em cluster} of $\Lb$ is, by definition,
a family $\Pb = (P_i)_{i\le m}$ where $P_i \subset \Lambda_i$ is 
finite for all $i\le m$. For a bounded set $A \subset \R^d$, let
$A \cap \Lb := (A \cap \Lam_i)_{i \le m}$. 
There is a natural translation $\R^d$-action on the set of Delone $m$-sets and their clusters
in $\R^d$. The translate of a cluster $\Pb$ by $x \in \R^d$ is $x + \Pb = (x+P_i)_{i \le m}$.
We say that two clusters $\Pb$ and $\Pb'$ are translationally equivalent if $\Pb = x + \Pb'$, i.e. 
$P_i = x + P_i'$ for all $i \le m$, for some $x \in \R^d$. 

Recall that  linear map $Q : \R^d \rightarrow \R^d$ is {\em expanding}
if its every eigenvalue lies outside the unit circle. 

\begin{definition} \label{def-subst-mul}
{ $\Lb = (\Lam_i)_{i\le m}$ is called a {\em
substitution Delone $m$-set} if $\Lb$ is a Delone $m$-set and
there exist an expanding map
$Q:\, \R^d\to \R^d$ and finite sets $\Dk_{ij}$ for $i,j\le m$ (possibly empty) such that
\be \label{eq-sub}
\Lambda_i = \biguplus_{j=1}^m (Q \Lambda_j + \Dk_{ij}),\ \ \ i \le m,
\ee
where $\biguplus$ denotes disjoint union.  The {\em substitution matrix}  $\Sf$ is defined by $\Sf_{ij} = \sharp(\Dk_{ij})$. The substitution $m$-set is {\em primitive} if $\Sf$ is primitive, i.e., some power of $\Sf$ has only strictly positive entries.}

With an abuse of terminology, we say that $\supp(\Lb)$, or sometimes even $\Lb$, is simply a substitution Delone set.
\end{definition}

There is a connection between substitution Delone sets and Delone sets with inflation symmetries, discussed in Section 3. Of course, here $Q$ is more general, whereas in Section 3 it was a homothety $\bx\mapsto \eta\bx$. It is not always true that $\supp(\Lb) \supset Q (\supp(\Lb))$; a sufficient condition is that for every $j\le m$ there exists $i\le m$ such that $\Dk_{ij} \ni 0$.
This may be achieved by passing from $\Lam_i$ to $\Lam_i + x_i$, which satisfy a system of equations as in (\ref{eq-sub}), with modified $\Dk_{ij}$; the only issue is whether the new
$m$-set is still a Delone $m$-set. In any case, the following holds:

\begin{lemma} \label{lem-alg} Suppose that $\Lb$ is a substitution Delone $m$-set with expansion map $Q$. If $\supp(\Lb)$ is finitely generated, then all eigenvalues of $Q$ are algebraic integers.
\end{lemma}

For the proof, it is convenient to consider the set of ``inter-atomic vectors''
\be \label{def-Xi}
\Xi(\Lb):= \bigcup_{i=1}^m (\Lam_i - \Lam_i).
\ee
It is immediate from (\ref{eq-sub}) that $Q(\Xi(\Lb))\subset \Xi(\Lb)$. Even though it need not be a Delone set (it is Delone when $\supp(\Lb)$ is a Meyer set), the proof proceeds similarly to the
proof of Lemma~\ref{lem-algeb}.

\medskip

There is another important necessary condition for $Q$ to be an expansion map.

\begin{theorem}[{\cite[Th.\,2.3]{LaWa}}] \label{th-PF}
If $\Lb$ is a primitive substitution Delone $m$-set with expansion map $Q$, then the Perron-Frobenius (PF) eigenvalue $\lam(\Sf)$ of the substitution matrix $\Sf$ equals $|\det(Q)|$.
\end{theorem}

In fact, it is not difficult to see that 
$$
\supp(\Lb)\ \ \mbox{is relatively dense}\ \Rightarrow\ \lam(\Sf) \ge |\det(Q)|;
$$
$$
\supp(\Lb)\ \ \mbox{is uniformly discrete}\ \Rightarrow\ \lam(\Sf) \ge |\det(Q)|.
$$

For each primitive substitution Delone $m$-set $\Lb$ (\ref{eq-sub}) one can set up  
an {\em adjoint system} of equations
\be \label{eq-til}
Q A_j = \bigcup_{i=1}^m (\Dk_{ij} + A_i),\ \ \ j \le m.
\ee
From the theory of 
graph-directed iterated function systems,
it follows that (\ref{eq-til}) always has a unique solution 
for which $\Ak = \{A_1, \dots, A_m\}$ is 
a family of non-empty compact sets of $\R^d$. 
It is proved in \cite[Th.\,2.4 and Th.\,5.5]{LaWa} that if $\Lb$ is a primitive
substitution Delone $m$-set, then all the sets $A_i$ from (\ref{eq-til})
have non-empty interiors and, moreover, each $A_i$ is the closure of 
its interior. From Theorem~\ref{th-PF} it follows that the interiors of the sets in the right-hand side of (\ref{eq-til}) are disjoint, hence we have a natural candidate for a tiling. We next review briefly the relevant tiling definitions.


\subsection{Tilings and substitution tilings}
This section has a significant overlap with \cite{FrankLN}.

We begin with a set of types (or colors) $\{1,\ldots,m\}$, 
which we fix once and for all. 
A {\em tile} in $\R^d$ is defined as a pair $T=(A,i)$ where $A=\supp(T)$ 
(the support of $T$) is a compact
set in $\R^d$ which is the closure of its interior, and 
$i=\ell(T)\in \{1,\ldots,m\}$
is the type of $T$. We let $g+T = (g+A,i)$ for $g\in \R^d$. Given a tile $T$ and a set $X\subseteq \R^d$ we use the notation:
$$
T + X = \{T+x:\ x\in X\}.
$$
A finite set $P$ of tiles is called a {\em patch} if the tiles of $P$ have mutually disjoint
interiors (strictly speaking, we have to say ``supports of tiles,'' but this
abuse of language should not lead to confusion). 
A {\em tiling} of $\R^d$ is a set $\Tk$ of tiles such that 
$\R^d = \bigcup \{\supp(T) : T \in \Tk\}$ and distinct tiles have disjoint 
interiors.
Given a tiling $\Tk$, finite sets of tiles of $\Tk$ are called 
$\Tk$-patches.

We always assume that any two $\Tk$-tiles with the same color are translationally equivalent.
(Hence there are finitely many $\Tk$-tiles up to translation.)

\begin{definition}\label{def-subst}
{ Let $\Ak = \{T_1,\ldots,T_m\}$ be a finite set of tiles in $\R^d$
such that $T_i=(A_i,i)$; we will call them {\em prototiles}.
Denote by $\Pk_{\Ak}$ the set of
patches made of tiles each of which is a translate of one of $T_i$'s.
We say that $\omega: \Ak \to \Pk_{\Ak}$ is a {\em tile-substitution} (or simply
{\em substitution}) with
expanding map $Q$ if there exist finite sets $\Dk_{ij}\subset \R^d$ for
$i,j \le m$, such that
\begin{equation}
\om(T_j)= \bigcup_{i=1}^m (T_i + \Dk_{ij}).
\label{subdiv}
\end{equation}
Since $\om(T_j)$ is a patch, it follows  that for all $j\le m$,
$$
Q A_j  = \bigcup_{i=1}^m (A_i+\Dk_{ij}),
$$
and the sets in the right-hand side have disjoint interiors.}
\end{definition}

The substitution (\ref{subdiv}) is extended to all translates of prototiles by
$\om(x+T_j)= Q x + \om(T_j)$, and to patches and tilings by
$\om(P)=\bigcup\{\om(T):\ T\in P\}$.
The substitution $\om$ can be iterated, producing larger and larger patches
$\om^k(T_j)$. As for a substitution Delone $m$-set,  we associate to $\om$ its $m \times m$ 
substitution matrix $\Sf$, with $\Sf_{ij}:=\sharp (\Dk_{ij})$.
The substitution $\om$ is  said to be primitive
if $\Sf$ is primitive.
The tiling $\Tk$ is called a fixed point of a substitution if $\om(\Tk) = \Tk$.

\begin{definition}{
A tiling is called {\em self-affine} if it is a fixed point of a primitive tile-substitution. Usually it is also assumed that the tiling has finite local complexity (FLC). If the expansion map $Q$ is a similitude, that is, $Q=\eta \Ok$ for some $\eta>1$ and an orthogonal linear transformation $\Ok$, then we say that the tiling is {\em self-similar}. For a self-similar tiling in $\R^2$ one also considers the {\em complex expansion factor} $\lam\in \C$,  $|\lam|>1$, by identifying the plane with $\C$ and the map $Q$ with $z\mapsto \lam z$.
}
\end{definition}

Notice that a fixed point of a substitution naturally defines a substitution Delone $m$-set, as follows: By definition, we can write $\Tk = \bigcup_{j=1}^m (T_j + \Lam_j)$ for some Delone sets $\Lam_j$.
Then we have
\begin{eqnarray*}
\bigcup_{i=1}^m (T_i + \Lam_i ) = \Tk = \om(\Tk) & = & \bigcup_{j=1}^m\bigl(\om(T_j) + Q\Lam_j\bigr) \\
& = & \bigcup_{j=1}^m \Bigl(\bigcup_{i=1}^m (T_i + \Dk_{ij}) + Q\Lam_j\Bigr)\\
& = & \bigcup_{i=1}^m \Bigl(T_i + \bigcup_{j=1}^m (Q\Lam_j + \Dk_{ij}) \Bigr).
\end{eqnarray*}
It follows that $\Lb = (\Lam_i)_{i=1}^m$ satisfies the system of equations (\ref{eq-sub}). In general, it is not necessarily true that $\Lam_i$ are disjoint, but we can ensure this, e.g., by taking
$T_j:= \wt{T}_j - c(\wt{T}_j)$, where $\wt{T}_j$ is a $\Tk$-tile of type $j$ and $c(\wt{T}_j)$ is a point chosen in its interior. Then we obtain that $\Lb$ is a substitution Delone $m$-set.

A natural question is when this procedure can be reversed.

\subsection{Representable Delone $m$-sets}

\begin{definition} { A Delone $m$-set $\Lb = (\Lam_i)_{i \le m}$ is called
{\em representable} (by tiles) for a tiling if there exists a set of prototiles 
$\Ak = \{T_i : i\le m\} $
so that
\be\label{eq-1}
\Lb + \Ak := \{x + T_i :\ x\in \Lambda_i,\ i \le m\} \ \ \ \mbox{is a tiling of}\ \ 
\R^d,
\ee
that is, $\R^d = \bigcup_{i\le m} \bigcup_{x\in \Lambda_i} (x + A_i)$ where $T_i = (A_i,i)$ for 
$i \le m$, and the sets in this union have disjoint interiors. 
In the case that $\Lb$ is a primitive substitution Delone $m$-set we will 
understand the term representable to mean relative to the tiles 
$T_i = (A_i,i)$, for $i\le m$, arising from the solution to the adjoint system (\ref{eq-til}). 
We call $\Lb + \Ak$ the associated tiling of $\Lb$.}
\end{definition}

\begin{definition}{
For a substitution Delone $m$-set $\Lb = (\Lambda_i)_{i \le m}$ satisfying (\ref{eq-sub}), define a matrix $\Phi=(\Phi_{ij})_{i,j=1}^m$ 
whose entries are finite (possibly empty) families of linear affine transformations on 
$\R^d$ given by
$$
\Phi_{ij} = \{f: x\mapsto Qx+a:\ a\in \Dk_{ij}\}\,.
$$
We define $\Phi_{ij}(\Xk) := \bigcup_{f\in \Phi_{ij}} f(\Xk)$ for a set $\Xk\subset \R^d$. For an $m$-set $(\Xk_i)_{i\le m}$ let
$$
\Phi\bigl((\Xk_i)_{i\le m}\bigr) = \Bigl(\bigcup_{j=1}^m \Phi_{ij}(\Xk_j)\Bigr)_{i\le m}\,.
$$
Thus $\Phi(\Lb)= \Lb$ by definition. We say that $\Phi$ is an {\em $m$-set substitution}.
}
\end{definition}

Let $\Lb$ be a substitution Delone $m$-set and $\Phi$ the associated $m$-set substitution.

\begin{definition}
{ Let $\Lb$ be a primitive substitution Delone $m$-set and let
$\Pb$ be a cluster of $\Lb$. The
cluster $\Pb$ will be called {\em legal} if it is a translate of a subcluster of
$\Phi^k(\{x_j\})$ for some $x_j \in \Lam_j$, $j \le m$ and $k \in \N$. (Here $\{x_j\}$ is an $m$-set which is empty in all coordinates other than $j$, for which it is a singleton.)
}
\end{definition}

\begin{lemma}\cite{LMS03}
Let $\Lb$ be a primitive substitution Delone $m$-set such that every 
$\Lb$-cluster is legal. Then $\Lb$ is repetitive. 
\end{lemma}

Not every substitution Delone $m$-set is representable (see \cite[Ex.\,3.12]{LMS03}), but the following theorem provides the sufficient condition for it.

\begin{theorem}\cite{LMS03}\label{legal-rep}
Let $\Lb$ be a repetitive primitive substitution Delone $m$-set. Then every 
$\Lb$-cluster is legal if and only if $\Lb$ is representable.
\end{theorem}

\begin{remark} {
In \cite[Lemma 3.2]{LaWa} it is shown that if $\Lb$ is a substitution Delone $m$-set,
then there is a finite $m$-set (cluster) $\Pb \subset \Lb$ for which
$\Phi^{n-1}(\Pb) \subset \Phi^n(\Pb)$ for $n \ge 1$ and 
$\Lb = \lim_{n \to \infty} \Phi^n (\Pb)$. We call such a $m$-set $\Pb$ 
a {\em generating $m$-set}.
Note that, in order to check that every $\Lb$-cluster is legal, 
we only need to see if some cluster that contains a finite generating 
$m$-set for $\Lb$ is legal.}  
\end{remark}

\subsection{Characterization of expansion maps} An important question, first raised by Thurston \cite{Thur}, is to
characterize which expanding linear maps may occur as expansion
maps for self-affine (self-similar) tilings.
It is pointed out in
\cite{Thur} that in one dimension, $\eta>1$ is an expansion factor
if and only if it is a Perron number (necessity follows from the
Perron-Frobenius theorem and sufficiency follows from a result of
Lind \cite{Lind} as in Example \ref{example}(ii)). In two dimensions, Thurston \cite{Thur} proved
that if $\lam$ is a  complex expansion factor of a self-similar
tiling, then $\lam$ is a {\em complex Perron number}, that is, an
algebraic integer whose Galois conjugates, other than $\ov{\lam}$,
are all less than $|\lam|$ in modulus.

The following theorem was
stated in \cite{Ken.thesis}, but complete proof was not available
until much later.

\begin{theorem} \cite{Ken.thesis,KS} \label{th-KS}
Let $\phi$ be a diagonalizable (over $\Compl$) expansion map on
$\R^d$, and let $\Tk$ be a self-affine tiling of $\R^d$ with
expansion $\phi$. Then

{\rm (i)} every eigenvalue of $\phi$ is an algebraic integer;

{\rm (ii)}
 if $\lam$ is an eigenvalue of $\phi$ of multiplicity $k$ and $\gam$ is
an algebraic conjugate of $\lam$, then either $|\gam| < |\lam|$, or $\gam$
is also an eigenvalue of $\phi$ of multiplicity greater or equal to $k$.
\end{theorem}

Here part (i) is included for completeness; it is a folklore result, proved similarly to Lemma~\ref{lem-algeb}.
Recently Theorem~\ref{th-KS} was finally extended to the general, not necessarily diagonalizable, case by J. Kwapisz \cite{Kwapisz}; we don't state precise ``generalized Perron'' conditions here, but refer the reader to his paper.
Basically, one has to take into account the multiplicity of Jordan blocks as well. It is conjectured that the necessary condition is also sufficient, at least, in the weaker form: one should be able to construct a self-affine tiling with expansion map $Q^n$ for some $n\in \N$. Sufficiency (in the stronger form) in dimension one follows from \cite{Lind}, as discussed in Section 3, Example (ii), and the construction in the two-dimensional self-similar case is found in \cite{Kenyon.construction}.
A natural approach to the conjecture would be to construct first a substitution Delone $m$-set with the desired inflation symmetry, and then apply  Theorem~\ref{legal-rep}. This way, the geometric shape of tiles comes from the adjoint equation, and one has more freedom.

\medskip

The starting point in the proofs of necessity of the Perron condition is defining the {\em control points} for the tiles \cite{Thur}.

\begin{definition} 
{ Let $\Tk$ be a fixed point of a primitive tile-substitution with expanding map $Q$. For each $\Tk$-tile $T$, fix a tile $\gamma T$ in the patch $\omega (T)$; 
choose $\gamma T$ with the same relative position for all tiles of the same 
type. This defines a map $\gamma : \Tk \to \Tk$ called the 
{\em tile map}. Then define the {\em control point} for a tile $T \in \Tk$ by 
\[ \{c(T)\} = \bigcap_{n=0}^{\infty} Q^{-n}(\gamma^n T).\]
}
\end{definition}

\noindent
The control points have the following properties:
\begin{itemize}
\item[(a)]\ $T' = T + c(T') - c(T)$, for any tiles $T, T'$ of the same type;
\item[(b)]\ $Q(c(T)) = c(\gamma T)$, for $T \in \Tk$.
\end{itemize}

It immediately follows from these properties that $\Lb = (\Lam_i)_{i\le m}$, where $\Lam_i = \{c(T):\ T\in \Tk\ \mbox{of type}\ i\}$, is a substitution Delone $m$-set. Moreover, $X:=\supp(\Lb) = \{c(T):\ T\in \Tk\}$ is a Delone set satisfying
$QX\subset X$. Thurston \cite{Thur} defined the address map $\phi:\,[X]\to \Z^s$, as in Section 1, and considered the induced action of the linear expanding map $Q$. See the references for the rest.

\subsection{Pseudo-self-affine tilings} We mention briefly another instance where the duality between substitution Delone $m$-sets and substitution tilings was useful.
The reader is referred to \cite[4.4.2]{FrankLN} and the original papers for more details.

\begin{definition}[\cite{BSJ}]
To a subset $F\subset \R^d$ and a tiling $\Tk$ of $\R^d$ we associate a $\Tk$-patch by $[F]^\Tk = \{T\in \Tk:\ \supp(T)\cap F\ne \es\}$. Let $\Tk$ and $\Tk_2$ be two tilings. Say that $\Tk_2$ is {\em locally derivable} (LD) from $\Tk_1$ with radius $R>0$ if for all $x,y\in \R^d$,
$$
[B_R(x)]^{\Tk_1} = [B_R(y)]^{\Tk_2} + (x-y)\ \Rightarrow\ [\{x\}]^{\Tk_2} = [\{y\}]^{\Tk_2}  + (x-y).
$$ 
If $\Tk_2$ is LD from $\Tk_1$ and $\Tk_1$ is LD from $\Tk_2$, the tilings are {\em mutually locally derivable} (MLD). 
\end{definition}

\begin{definition} Let $Q:\,\R^d\to \R^d$ be an expanding linear map.
A repetitive FLC tiling of $\R^d$ is called a {\em pseudo-self-affine tiling} with expansion $Q$ if $\Tk$ is LD from $Q\Tk$. If $Q$ is a similitude, the pseudo-self-affine tiling is called {\em pseudo-self-similar}.
\end{definition}

E. A. Robinson, Jr.\ \cite{Robi} conjectured that every pseudo-self-affine tiling is MLD with a self-affine tiling. This was settled in the affirmative: for pseudo-self-similar tilings in $\R^2$ in \cite{FrankSol}, and  in \cite{SolPseudo} in full generality. A few comments:

\begin{itemize}

\item In \cite{FrankSol} the method of ``redrawing the boundary'' was used (following \cite{Kenyon.construction} to some extent); as a result we obtained an MLD self-similar tiling where each tile is a topological disk bounded by a Jordan curve.

\item In contrast, in \cite{SolPseudo} we first constructed an MLD substitution Delone $m$-set in $\R^d$, and then applied Theorem~\ref{legal-rep}. This way we only know that each tile is a compact set, which is a closure of its interior; the tiles need not even be connected.

\item In both papers \cite{FrankSol,SolPseudo} we needed to pass from the expansion $Q$ for the original tiling to the expansion $Q^k$ for $k$ sufficiently large for the resulting self-similar (or self-affine) tiling, in order for the construction to work, similarly to the weaker form of the conjecture, discussed after  Theorem~\ref{th-KS}.

\end{itemize}


\section{Dynamical systems from Delone sets in $\R^d$}

This section is not intended to be a comprehensive survey (even in the ``local'' sense); we only touch on some aspects of this topic. For more  on the background and other topics see \cite{FrankLN} and references therein.


\medskip

Let $\Lb$ be an FLC Delone $m$-set and let $X_{\Lbs}$ be the collection of all 
Delone $m$-sets each of whose
clusters is a translate of a $\Lb$-cluster. We introduce a ``big ball'' metric
on $X_{\Lbs}$ in the standard way: two Delone $m$-sets are close if the agree exactly in a large neighborhood of the origin, possibly after a small translation. Precisely:
\be \label{metric-multisets}
\rho(\Lb_1,\Lb_2) := \min\{\tilde{\rho}(\Lb_1,\Lb_2), 2^{-1/2}\}\, ,
\ee
where
\begin{eqnarray} 
\tilde{\rho}(\Lb_1,\Lb_2)
&=&\mbox{inf} \{ \e > 0 : \exists~ x,y \in B_{\e}(0), \nonumber \\ \nonumber
&  & ~~~~~~~~~~ B_{1/{\e}}(0) \cap (-x + \Lb_1) = B_{1/{\e}}(0) 
\cap (-y + \Lb_2) \}\,. 
\end{eqnarray}
For the proof that $\rho$ is a metric, see \cite{LMS1}. 
It is well-known that $(X_{\Lbs},\rho)$ is compact (here the FLC assumption is crucial).
If FLC fails, the ``big ball'' metric is modified, basically considering the Hausdorff distance between the sets $B_{1/\e}(0)\cap \Lb_1$ and $B_{1/\e}(0)\cap \Lb_2$. Here we restrict ourselves to the FLC case.

\smallskip

Observe that $X_{\Lbs} = \overline{\{-h + \Lb : h \in \R^d \}}$ where the closure is taken in the topology induced by the metric $\rho$. 
The group $\R^d$ acts on $X_{\Lbs}$ by translations which are homeomorphisms, and we get a topological dynamical system $(X_{\Lbs},\R^d)$.
We should point out that most of the definitions and statements in this section have a parallel version in the tiling setting. It is usually not a problem to pass from the tiling framework to the Delone set framework and back, and we will do so freely.

Certain discrete-geometric and statistical properties of the Delone $m$-set $\Lb$ correspond to properties of the associated dynamical system. For instance, $\Lb$ is repetitive if and only if $(X_{\Lbs},\R^d)$ is minimal (i.e., every
orbit is dense), see \cite{Robi}.

We next discuss cluster (patch) frequencies and { invariant measures}.
For a cluster $\Pb$ and a bounded set $A\subset \R^d$ denote
$$
L_{\Pbs}(A) = \sharp\{x\in \R^d:\ x+\Pb \subset A\cap \Lb\},
$$
where $\sharp$ denotes the cardinality.
In plain language, $L_{\Pbs}(A)$ is the number of translates of $\Pb$
contained in $A$, which is clearly finite.
For a bounded set $F \subset \R^d$ and $r > 0$, let $(F)^{+r}:=
\{x \in \R^d:\,\dist(x,F) \le r\}$ denote the $r$-neighborhood of $F$.
A {\em van Hove sequence} for $\R^d$ is a sequence 
$\mathcal{F}=\{F_n\}_{n \ge 1}$ of bounded measurable subsets of 
$\R^d$ satisfying
\be \label{Hove}
\lim_{n\to\infty} \Vol((\partial F_n)^{+r})/\Vol(F_n) = 0,~
\mbox{for all}~ r>0.
\ee

\begin{definition} \label{def-ucf}{ 
Let $\{F_n\}_{n \ge 1}$ be a van Hove sequence.
The Delone $m$-set $\Lb$ has {\em uniform cluster frequencies} (UCF)
(relative to $\{F_n\}_{n \ge 1}$) if for any non-empty cluster $\Pb$, the limit
$$
\freq(\Pb,\Lb) = \lim_{n\to \infty} \frac{L_{\Pbs}(x+F_n)}{\Vol(F_n)} \ge 0
$$
exists uniformly in $x\in \R^d$.}
\end{definition}

Recall that a topological dynamical system is {\em uniquely ergodic}
if there is a unique invariant probability measure (which is then automatically
ergodic).
It is known (see e.g. \cite[Th.\ 2.7]{LMS1})
that for a Delone $m$-set $\Lb$ with FLC,
the dynamical system $(X_{\Lbs},\R^d)$ is uniquely ergodic
if and only if $\Lb$ has UCF.
A primitive FLC substitution Delone $m$-set is known to have UCF (see \cite{LMS03}), hence we get a uniquely ergodic $\R^d$-action $(X_{\Lbs},\R^d,\mu)$.


\subsection{Eigenvalues of Delone set dynamical systems}

Let $\mu$ be an ergodic invariant Borel probability
measure for the dynamical system 
$(X_{\Lbs},\R^d)$.
A point $\alpha =(\alpha_1,\ldots,\alpha_d) \in \R^d$ is called 
an eigenvalue (or dynamical eigenvalue) for the $\R^d$-action if there exists an eigenfunction
$f\in L^2(X_{\Lbs},\mu),$ that is, $\ f\not\equiv 0$ and
\be 
\label{eq-eigen} f(-g+\Sk)= e^{2 \pi i \langle g, \alpha\rangle} f(\Sk),\ \ \ \mbox{for all}
\ \ g\in \R^d.
\ee
Here $\langle g, \alpha\rangle$ is the usual scalar product in $\R^d$ and the equality is understood in $L^2$, that is, for $\mu$-a.e.\ $\Sk$.

An eigenvalue $\alpha$ is  a {\em continuous} or {\em topological} eigenvalue if (\ref{eq-eigen}) has a continuous solution. A characterization of topological eigenvalues was obtained in \cite{SolDelone}.

\begin{definition}
Let $\Lb$ be a Delone $m$-set of finite type. We say that $\by \in \R^d$ is a {\em topological $\delta$-almost-period} for $\Lb$ if
$$
\Lb\cap B_{1/\delta}(0) = (\Lb-y) \cap B_{1/\delta}(0).
$$
Denote by $\Psi_\delta(\Lb)$ the set of topological $\delta$-almost-periods.
\end{definition}

\begin{theorem}[\cite{SolDelone}] \label{th-topeigen}
Let $\Lb$ be a repetitive Delone $m$-set of finite type.  Then $\alpha$ is a topological eigenvalue for $(X_{\Lbs},\R^d)$ if and only if
\be \label{eq-topeigen}
\lim_{\delta\to 0} \sup_{y\in \Psi_\delta(\Lbs)} |e^{2 \pi i \langle y, \alpha\rangle}-1|=0.
\ee
\end{theorem}

In \cite{SolDelone} the theorem is proved for a single Delone set dynamical system, but the proof transfers to the case of Delone $m$-sets without any changes.

\medskip

There is a connection between the {\em diffraction spectrum} of a Delone $m$-set and the dynamical spectrum, but this topic is beyond the scope of these Notes. On this matter, the reader should consult e.g., \cite{BG,BL} and references therein.

N. Strungaru \cite{Str} proved that any Meyer set has a relatively dense of Bragg peaks, which implies, via the link with the dynamical spectrum, that the set of eigenvalues for the associated dynamical system is relatively dense as well. More recently, J. Kellendonk and L. Sadun \cite{KeSa14} proved that the latter property holds for topological eigenvalues as well. In fact, they established the following

\begin{theorem}[{\cite[Th.\,1.1]{KeSa14}}] \label{th-LS}
A repetitive FLC Delone set dynamical system in $\R^d$ has $d$ linearly independent topological eigenvalues if and only if it is topologically conjugate to a Meyer set dynamical system.
\end{theorem}

In the case of substitution systems, we obtained the following result earlier, jointly with J.-Y. Lee:

\begin{theorem}[{\cite[Th.\,4.14]{LS08}}] Let $\Lb=(\Lam_j)_{j=1}^m$ be a representable primitive FLC substitution Delone $m$-set. The set of eigenvalues for the $\R^d$-action $(X_{\Lbs},\R^d,\mu)$ is relatively dense in $\R^d$ if and only if $\supp(\Lb) = \bigcup_{j=1}^m \Lam_j$ is a Meyer set.
\end{theorem}

As a corollary, we showed in \cite{LS08} that if the $\R^d$-action $(X_{\Lbs},\R^d,\mu)$ has purely discrete spectrum, then $\supp(\Lb)$ is a Meyer set. This was an answer to a question of J. Lagarias from \cite{Lag00}.

For the proof of sufficiency of the Meyer property we relied on the result of Strungaru \cite{Str} quoted above. The proof of necessity in Theorem~\ref{th-LS} proceeds via the notion of {\em Pisot families}. Let $Q$ be the expansion map for the substitution Delone set. By Lemma~\ref{lem-alg},  the set of eigenvalues of $Q$ consists of algebraic integers.
Following Mauduit \cite{Maud}, we
say that a set $\FrP$ of algebraic integers forms a {\em Pisot family} if for every $\lam\in \FrP$
 and every Galois conjugate $\lam'$ of $\lam$, if
$\lam'\not\in \FrP$, then $|\lam'| < 1$. 

The link from eigenvalues to Number Theory comes from the following, which we restate here in different terms:

\begin{theorem} \cite[Th.\,4.3]{soltil} \label{translation-vector-formula}
Let $\Lb$ be a repetitive substitution Delone $m$-set with expansion map $Q$, which has FLC. Let $\Xi(\Lb)$ be the set of ``inter-atomic'' vectors defined in (\ref{def-Xi}).
If $\alpha \in \R^d$ is an eigenvalue for $(X_{\Lbs}, \R^d, \mu)$, 
then for any $x \in \Xi(\Lb)$ we have 
$\|\langle Q^n x, \alpha \rangle\| \to 0$ as $n\to \infty$.
\end{theorem}

Here  $\|t\|$ denotes the distance from $t$ to the nearest integer.
Then we apply the following result, a generalization of the classical Pisot's theorem, which we only partially state:

\begin{theorem}[{\cite{Korn,Maud}}]
Let $\lam_1, \dots, \lam_r$ be distinct algebraic numbers such that $|\lambda_i| \ge 1$, $i = 1, \dots, r$, and let
$P_1, \dots, P_r$ be nonzero polynomials with complex coefficients. 
If $\sum_{i=1}^r P_i (n) \lam_i^n$ is real for all $n$ and 
\[ \lim_{n \to \infty} \Bigl\| \sum_{i=1}^r P_i (n) \lam_i^n\Bigr\| = 0,\]
Then $\{\lam_1,\ldots,\lam_r\}$ is a Pisot family.
\end{theorem} 



%

%


%

%
\begin{acknowledgement}
Thanks to CIRM, Luminy, for hospitality and for providing a perfect work environment. 
The author is grateful to Shigeki Akiyama, the 2017 Morlet Chair, and to Pierre Arnoux for the invitation and for running a very successful and stimulating program. The research of B.S. was supported by the Israel Science Foundation (Grant 396/15).
\end{acknowledgement}
%

%
%
%

\end{document}